\documentclass[a4paper,12pt]{amsart}
\usepackage[margin=3cm]{geometry}
\newtheorem{theorem}{Theorem}[section]
\newtheorem{rem} [theorem] {Remark}
\newtheorem{prop} [theorem]{Proposition}
\newtheorem{lemma}[theorem]{Lemma}
\newtheorem{definition}[theorem]{Definition}
\newcommand{\ovprt}{\overline{\partial}}
\newcommand{\ovli}{\overline}
\newcommand{\dquer}{\overline\partial}
\newcommand{\dquers}{\overline\partial ^*_\varphi}
\newcommand{\boxphi}{\square_\varphi}
\newcommand{\levim}{\frac{\partial^2\varphi}{\partial z_j\partial\overline z_k}}

\numberwithin{equation}{section}
\begin{document}
\title{Compactness estimates for  the  $\ovprt $ - Neumann problem  in weighted $L^2$ - spaces.}

\author{Klaus Gansberger and Friedrich Haslinger\\
 \  \\
 {\it {\tiny Dedicated to Linda Rothschild}}}

\thanks{Partially supported by the FWF-grant  P19147.}

 \address{K. Gansberger, F. Haslinger: Institut f\"ur Mathematik, Universit\"at Wien,
Nordbergstrasse 15, A-1090 Wien, Austria}
\email{klaus.gansberger@univie.ac.at , friedrich.haslinger@univie.ac.at}

\maketitle


\begin{abstract} ~\\
In this paper we discuss compactness estimates for the $\ovprt $-Neumann problem in the setting of weighted $L^2$-spaces on $\mathbb{C}^n.$
For this purpose we use  a version of the Rellich - Lemma for weighted Sobolev spaces. 
\end{abstract}

\section{Introduction.}~\\

Let $\Omega $ be a bounded pseudoconvex domain in $\mathbb C^n.$ 
We consider the 
$\ovprt $-complex 
$$
L^2(\Omega )\overset{\ovprt }\longrightarrow L^2_{(0,1)}(\Omega)
\overset{\ovprt }\longrightarrow \dots \overset{\ovprt }\longrightarrow
L^2_{(0,n)}(\Omega)\overset{\ovprt }\longrightarrow 0\, ,  
$$
where $L^2_{(0,q)}(\Omega)$ denotes the space of $(0,q)$-forms on $\Omega$ with
coefficients in $L^2(\Omega).$ The $\ovprt $-operator on $(0,q)$-forms is given by
$$\ovprt \left ( \sum_J\,^{'}  a_J \, d\ovli z_J \right )=
\sum_{j=1}^n \sum_J\,^{'}\  \frac{\partial a_J}{\partial \ovli z_j}d\ovli z_j\wedge
d\ovli z_J,$$
where $\sum  ^{'} $ means that the sum is only taken over strictly increasing multi-indices $J.$

The derivatives are taken in the sense of distributions, and the domain of $\ovprt $
consists of those $(0,q)$-forms for which the right hand side belongs to
$L^2_{(0,q+1)}(\Omega).$ So $\ovprt $ is a densely defined closed operator, and
therefore has an adjoint operator from $L^2_{(0,q+1)}(\Omega)$ into
$L^2_{(0,q)}(\Omega)$ denoted by $\ovprt ^* .$

The complex Laplacian 
$\Box = \ovprt \, \ovprt ^* + \ovprt ^* \, \ovprt $
acts as
an unbounded selfadjoint  operator on 
$L^2_{(0,q)}(\Omega ),\ 1\le q \le n,$
it is surjective and therefore
has a continuous inverse, the $\ovprt $-Neumann operator $N_q.$ If $v$ is a $\ovprt $-closed
$(0,q+1)$-form, then $u=\ovprt ^* \, N_{q+1}v$ provides the canonical solution to 
$\ovprt u = v,$ namely the one orthogonal to the kernel of $\ovprt$ and
so the one with minimal norm (see for instance \cite{ChSh}).

\vskip 0.5 cm

A survey of the $L^{2}$-Sobolev theory
of the $\overline{\partial}$-Neumann problem is given in \cite{BS}.
\vskip 0.5 cm

The question of compactness of $N_q$ is of interest for various
reasons. For example, compactness of $N_q$ implies global regularity
in the sense of preservation of Sobolev spaces \cite{KN}. Also, the
Fredholm theory of Toeplitz operators is an immediate consequence
of compactness in the $\overline{\partial}$-Neumann problem 
\cite{HI}, \cite{CD}. There are additional ramifications for certain
$C^{*}$-algebras naturally associated to a domain in $\mathbb{C}^{n}$
\cite{SSU}. Finally, compactness is a more robust property than
global regularity - for example, it localizes, whereas global
regularity does not - and it is generally believed to be more
tractable than global regularity. 

A thorough discussion of
compactness in the $\overline{\partial}$-Neu\-mann problem can be found
in \cite{FS1} and \cite{FS2}.

The study of the $\overline{\partial}$-Neumann problem is essentially
equivalent  to the study of the
canonical solution operator to $\overline{\partial}$:

The $\ovprt $-Neumann operator $N_q$ is compact from $L^2_{(0,q)}(\Omega )$ to itself if and only if the canonical solution operators
$$\ovprt ^*N_q: L^2_{(0,q)}(\Omega )\longrightarrow L^2_{(0,q-1)}(\Omega )   \ \ \textrm{and} \ \
\ovprt ^*N_{q+1}: L^2_{(0,q+1)}(\Omega )\longrightarrow L^2_{(0,q)}(\Omega )$$
are compact.

Not very much is known in the case of unbounded domains.
 In this paper we continue the investigations of \cite{HaHe} concerning  existence and compactness of the canonical solution operator to $\ovprt $ on weighted
$L^2$-spaces over $\mathbb C^n,$ 
where we applied ideas which were used in the spectral analysis of the Witten Laplacian in the real case, see \cite{HeNi}.

Let $\varphi : \mathbb C^n \longrightarrow \mathbb R^+ $ be a plurisubharmonic $\mathcal C^2$-weight function and define the space
$$L^2(\mathbb C^n , \varphi )=\{ f:\mathbb C^n \longrightarrow \mathbb C \ : \ \int_{\mathbb C^n}
|f|^2\, e^{-\varphi}\,d\lambda < \infty \},$$
where $\lambda$ denotes the Lebesgue measure, the space $L^2_{(0,1)}(\mathbb C^n, \varphi )$ of $(0,1)$-forms with coefficients in
$L^2(\mathbb C^n , \varphi )$ and the space $L^2_{(0,2)}(\mathbb C^n, \varphi )$ of $(0,2)$-forms with coefficients in
$L^2(\mathbb C^n , \varphi ).$
Let 
$$\langle f,g\rangle_\varphi=\int_{\mathbb{C}^n}f \,\overline{g} e^{-\varphi}\,d\lambda$$
denote the inner product and 
$$\| f\|^2_\varphi =\int_{\mathbb{C}^n}|f|^2e^{-\varphi}\,d\lambda $$
the norm in $L^2(\mathbb C^n , \varphi ).$

We consider the weighted
$\ovprt $-complex 
$$
L^2(\mathbb C^n , \varphi )\underset{\underset{\ovprt_\varphi^* }
\longleftarrow}{\overset{\ovprt }
{\longrightarrow}} L^2_{(0,1)}(\mathbb C^n , \varphi )
\underset{\underset{\ovprt_\varphi^* }
\longleftarrow}{\overset{\ovprt }
{\longrightarrow}} L^2_{(0,2)}(\mathbb C^n , \varphi ),
$$
where $\ovprt_\varphi^*$ is the adjoint operator to $\ovprt $ with respect to the weighted inner product. For $u=\sum_{j=1}^nu_jd\overline z_j\in {\text {dom}}(\ovprt_\varphi^*)$ one has
$$\ovprt_\varphi^*u=-\sum_{j=1}^n \left ( \frac{\partial}{\partial z_j}-
\frac{\partial \varphi}{\partial z_j}\right )u_{j}.$$
The complex Laplacian on $(0,1)$-forms is defined as
$$\boxphi := \dquer  \,\dquers + \dquers \dquer,$$
where the symbol $\boxphi $ is to be understood as the maximal closure of the operator initially defined on forms with coefficients in $\mathcal{C}_0^\infty$, i.e., the space of smooth functions with compact support.

$\boxphi $ is a selfadjoint and positive operator, which means that 
$$\langle \boxphi f,f\rangle_\varphi \ge 0 \ , \   {\text{for}} \  f\in dom (\boxphi ).$$
The associated Dirichlet form is denoted by 
$$Q_\varphi (f,g)= \langle \dquer f,\dquer g\rangle_\varphi + \langle \dquers f ,\dquers g\rangle_\varphi, $$
for $f,g\in dom (\dquer ) \cap dom (\dquers ).$ The weighted $\dquer $-Neumann operator
$N_\varphi $ is - if it exists - the bounded inverse of $\boxphi .$ 
\vskip 0.5 cm
There is an interesting  connection between $\ovprt $ and the theory of Schr\"odinger operators with magnetic fields,
see for example \cite{Ch}, \cite{B}, \cite{FS3} and \cite{ChF} for recent contributions exploiting this point of view.

Here we use a Rellich - Lemma for weigthed Sobolev spaces to establish compactness estimates for the $\dquer $-Neumann operator $N_\varphi$ on $L^2_{(0,1)}(\mathbb{C}^n, \varphi)$ and we use this to
give a new proof of the main result of \cite{HaHe} without  spectral theory of Schr\"odinger operators.
\vskip 1 cm

\section{Weighted basic estimates.}~\\
In the weighted space $L^2_{(0,1)}(\mathbb C^n, \varphi)$ we can give a simple characterization of dom\,($\dquers$):

\begin{prop}
\label{dom dquers}
Let $f=\sum f_jd\overline z_j\in L^2_{(0,1)}(\mathbb C^n, \varphi)$. Then $f\in dom(\dquers)$ if and only if 
\begin{equation*}
 \sum_{j=1}^n\left ( \frac{\partial f_j}{\partial z_j}- \frac{\partial \varphi}{\partial z_j}\, f_j \right )
 \in L^2(\mathbb{C}^n, \varphi ).
\end{equation*}
\end{prop}

\begin{proof}
Suppose first that $\sum_{j=1}^n\left ( \frac{\partial f_j}{\partial z_j}- \frac{\partial \varphi}{\partial z_j}\, f_j \right )
 \in L^2(\mathbb{C}^n, \varphi ),$ which equivalently means that $e^{\varphi }\sum_{j=1}^n\frac{\partial}{\partial z_j}\left( f_je^{-\varphi}\right) \in L^2(\mathbb{C}^n, \varphi )$. We have to show that there exists  a constant $C$ such that $|\langle\dquer g,f\rangle_{\varphi}|\leq C\Vert g\Vert_\varphi$ for all $g\in dom(\dquer)$. To this end let $(\chi_R)_{R\in\mathbb N}$ be a family of radially symmetric smooth cutoff funtions, which are identically one on $\mathbb B_R$, the ball with radius $R$, such that the support of $\chi _R$ is contained in $\mathbb B_{R+1}$,  $supp(\chi_R)\subset \mathbb B_{R+1}$, and such that furthermore all first order derivatives of all functions in this family are uniformly bounded by a constant $M$. Then for all $g\in \mathcal{C}_0^\infty (\mathbb{C}^n)$: 
\begin{equation*}
\langle\dquer g, \chi _Rf\rangle_\varphi=\sum_{j=1}^n\langle\frac{\partial g}{\partial\overline z_j}, \chi _Rf_j\rangle_\varphi=-\int\sum_{j=1}^n g\frac{\partial}{\partial \overline z_j}\left(\chi_R \overline f_je^{-\varphi}\right)d\lambda,
\end{equation*}
by integration by parts, which in particular means
\begin{equation*}
|\langle\dquer g,f\rangle_\varphi|=\lim_{R\to\infty}|\langle\dquer g, \chi _Rf\rangle_\varphi |
=\lim_{R\to\infty}\left|\int_{\mathbb{C}^n}\sum_{j=1}^n g\frac{\partial}{\partial \overline z_j}\left(\chi_R \overline f_je^{-\varphi}\right)\, d\lambda\right|.
\end{equation*}
 Now we use the triangle inequality, afterwards Cauchy -- Schwarz, to get
\begin{align*}
&\lim_{R\to\infty}\left|\int_{\mathbb{C}^n}\sum_{j=1}^n g\frac{\partial}{\partial \overline z_j}\left(\chi_R \overline f_je^{-\varphi}\right)\,d\lambda\right| \\
\leq&\lim_{R\to\infty}\left|\int_{\mathbb{C}^n}\chi_R \, g\sum _{j=1}^n\frac{\partial}{\partial\overline z_j}\left(\overline f_je^{-\varphi}\right)\,d\lambda\right|+\lim_{R\to\infty}\left|\int_{\mathbb{C}^n} \sum _{j=1}^n\overline f_jg\frac{\partial\chi_R}{\partial\overline z_j}e^{-\varphi}\,d \lambda\right|\\
\leq&\lim_{R\to\infty}\Vert \chi_R\, g\Vert_\varphi\left\Vert e^{\varphi}\sum_{j=1}^n\frac{\partial}{\partial z_j}\left( f_je^{-\varphi}\right)\right\Vert_\varphi+M\Vert g\Vert_\varphi\Vert f\Vert_\varphi\\
=&\Vert g\Vert_\varphi\left\Vert e^\varphi\sum_{j=1}^n\frac{\partial}{\partial z_j}\left( f_je^{-\varphi}\right)\right\Vert_\varphi+M\Vert g\Vert_\varphi\Vert f\Vert_\varphi.
\end{align*}
Hence by assumption,
\begin{equation*}
|\langle\dquer g,f\rangle_\varphi|\le\Vert g\Vert_\varphi\left\Vert e^\varphi\sum_{j=1}^n\frac{\partial}{\partial z_j}\left( f_je^{-\varphi}\right)\right\Vert_\varphi+M\Vert g\Vert_\varphi\Vert f\Vert_\varphi
\leq C\Vert g\Vert_\varphi
\end{equation*}
for all $g\in \mathcal{C}_0^\infty(\mathbb{C}^n)$, and by density of $\mathcal{C}_0^\infty(\mathbb{C}^n)$ this is true for all $g\in dom(\dquer)$. Conversely, let $f\in dom(\dquer^*_\varphi)$ and take $g\in \mathcal{C}_0^\infty(\mathbb{C}^n)$. Then $g\in dom(\dquer)$ and
 \begin{align*}
 \langle g,\dquer^*_\varphi f \rangle_\varphi=&\langle\dquer g,f\rangle_\varphi\\
 =&\sum_{j=1}^n\langle\frac{\partial g}{\partial \overline z_j},f_j\rangle_\varphi\\
 =&-\langle g,\sum_{j=1}^n\frac{\partial}{\partial z_j}\left( f_je^{-\varphi}\right)\rangle_{L^2}\\
 =&-\langle g,e^{\varphi}\sum _{j=1}^n\frac{\partial}{\partial z_j}\left( f_je^{-\varphi}\right)\rangle_\varphi.
 \end{align*}
Since $\mathcal{C}_0^\infty(\mathbb{C}^n)$ is dense in $L^2(\mathbb{C}^n,\varphi)$, we conclude that 
\begin{equation*}
\dquer^*_\varphi f=-e^{\varphi}\sum_{j=1}^n\frac{\partial}{\partial z_j}\left( fe^{-\varphi}\right),
\end{equation*}
which in particular implies that $e^\varphi\sum_{j=1}^n\frac{\partial}{\partial z_j}\left( f_je^{-\varphi}\right)\in L^2(\mathbb C^n,\varphi)$.
\end{proof}
 
The following Lemma will be important for our considerations.

\begin{lemma}
\label{density}
Forms with coefficients in $\mathcal{C}_0^\infty(\mathbb{C}^n)$ are dense in $dom(\dquer)\cap dom(\dquers)$ in the graph norm $f\mapsto (\Vert f\Vert^2 _\varphi+\Vert \dquer f\Vert _\varphi^2+\Vert \dquers f\Vert _\varphi^2)^\frac{1}{2}$.
\end{lemma}

\begin{proof}
First we show that compactly supported $L^2$-forms are dense in the graph norm. So let $\lbrace\chi_R\rbrace_{R\in \mathbb N}$ be a family of smooth radially symmetric cutoffs identically one on $\mathbb B_R$ and supported in $\mathbb B _{R+1}$, such that all first order derivatives of the functions in this family are uniformly bounded in $R$ by a constant $M$.\\
Let $f\in dom(\dquer)\cap dom(\dquers)$. Then, clearly, $\chi_R f\in dom(\dquer)\cap dom(\dquers)$ and $\chi_R f\to f$ in $L^2_{(0,1)}(\mathbb C^n, \varphi)$ as $R\to\infty$. As observed in Proposition \ref{dom dquers}, we have
\begin{equation*}
\dquer^*_\varphi f=-e^{\varphi}\sum_{j=1}^n\frac{\partial}{\partial z_j}\left( f_je^{-\varphi}\right),
\end{equation*}
hence
\begin{equation*}
\dquers (\chi_Rf)=-e^{\varphi}\sum_{j=1}^n\frac{\partial}{\partial z_j}\left( \chi _Rf_je^{-\varphi}\right).
\end{equation*}
We need to estimate the difference of these expressions 
\begin{equation*}
\dquers f-\dquers (\chi_R f) =\dquers f -\chi_R\dquers f +\sum _{j=1}^n\frac{\partial \chi_R}{\partial z_j}f_j,
\end{equation*}
which is by the triangle inequality
\begin{align*}
\Vert \dquers f-\dquers (\chi_R f)\Vert_\varphi \leq&\Vert\dquers f -\chi_R\dquers f\Vert_\varphi+M\sum_{j=1}^n\int \limits_{\mathbb C^n\setminus\mathbb B_R}|f_j|^2e^{-\varphi}\, d\lambda.
\end{align*}
Now both terms tend to $0$ as $R\to\infty$, and one can see similarly that also $\dquer(\chi_R f)\to\dquer f$ as $R\to \infty$.\\
So we have density of compactly supported forms in the graph norm, and density of forms with coefficients in $\mathcal{C}_0^\infty(\mathbb{C}^n)$ will follow by applying Friedrich$^\prime$s Lemma, see appendix D in \cite{ChSh}, see also \cite{Jo}.
\end{proof}

As in the case of bounded domains, the canonical solution operator to $\dquer,$ which we  denote by $\mathcal{S}_\varphi,$ is given by $\dquers N_\varphi.$ Existence and compactness of $N_\varphi $ and $\mathcal{S}_\varphi $ are closely related. At first, we notice that equivalent weight functions have the same properties in this regard.

\begin {lemma}
\label{glattheit}
Let $\varphi _1$ and $\varphi _2$ be two equivalent weights, i.e., $C^{-1}\Vert .\Vert _{\varphi _1}\leq\Vert .\Vert_{\varphi _2}\leq C\Vert .\Vert _{\varphi _1}$ for some $C>0$. Suppose that $\mathcal S_{\varphi_2}$ exists. Then $\mathcal S_{\varphi_1}$ also exists and  $\mathcal S_{\varphi_1}$ is compact if and only if $\mathcal S _{\varphi_2}$ is compact.\\ 
An analog statement is true for the weighted $\dquer$-Neumann operator.
\end{lemma}

\begin{proof}  Let $\iota$ be the identity $\iota : L^2_{(0,1)}(\mathbb C^n,\varphi_1)\rightarrow L^2_{(0,1)}(\mathbb C^n,\varphi_2)$, $\iota f =f$, let $j$ be the identity $j: L^2_{\varphi_2}\to L^2_{\varphi_1}$ and let furthermore $P$ be the orthogonal projection onto $\ker(\dquer)$ in $L^2_{\varphi_1}$ . Since the weights are equivalent, $\iota$ and $j$ are continuous, so if $\mathcal S_{\varphi_2}$ is compact, $j\circ\mathcal S_{\varphi_2}\circ \iota$ gives a solution operator on $L^2_{(0,1)}(\mathbb C^n,\varphi_1)$ that is compact. Therefore the canonical solution operator $\mathcal S_{\varphi _1}=P\circ j^{-1}\circ\mathcal S_{\varphi_2}\circ \iota$ is also compact. Since the problem is symmetric in $\varphi _1$ and $\varphi _2$, we are done.\\
The assertion for the Neumann operator follows by the identity 
$$N_\varphi=\mathcal S_\varphi\mathcal S_\varphi^*+\mathcal S_\varphi^*\mathcal S_\varphi.$$
\end{proof}

Note that whereas existence and compactness of the weighted $\dquer$-Neumann operator is invariant under equivalent weights by Lemma \ref{glattheit}, regularity is not. For examples on bounded pseudoconvex domains, see for instance \cite{ChSh}, chapter 6.\\
\vskip 0.5 cm
Now we suppose that the lowest eigenvalue $\lambda_\varphi $ of the Levi - matrix 
$$M_\varphi=\left(\levim\right)_{jk}$$
of $\varphi $ satisfies
$$\liminf_{|z|\to \infty}\lambda_\varphi (z)>0. \ \ \ (^*)$$
Then, by Lemma \ref{glattheit}, we may assume without loss of generality that $\lambda_\varphi(z)>\epsilon$ for some $\epsilon>0$ and all $z\in\mathbb C^n$, since changing the weight function on a compact set does not influence our considerations.  So we have the following basic estimate

\begin{prop}\label{basic}~\\ For a  plurisubharmonic weight function $\varphi $ satisfying (*), there is a $C>0$ such that
$$\|u\|^2_\varphi \le C( \| \ovprt u \|^2_\varphi + \| \ovprt_\varphi ^* u\|^2_\varphi )$$
for each $(0,1)$-form $u\in $dom\,$(\ovprt ) \,\cap$
dom\,$(\ovprt_\varphi^*).$
\end{prop}
\begin{proof}
By Lemma \ref{density} and the assumption on $\varphi$ it suffices to show that
$$\int_{\mathbb{C}^n} \sum_{j,k=1}^n\frac{\partial^2 \varphi}{\partial z_j \partial\overline{z}_k}\,u_j\overline{u}_k\, e^{-\varphi}\,d\lambda \le \|\overline{\partial}u\|^2_\varphi +
\| \ovprt_\varphi ^* u\|^2_\varphi ,$$
for each $(0,1)$-form $u=\sum_{k=1}^n u_k\,d\overline{z}_k$ with coefficients $u_k\in \mathcal{C}^\infty_0(\mathbb{C}^n), $ for $k=1,\dots ,n.
$

For this purpose we set $\delta_k=\frac{\partial}{\partial z_k}-\frac{\partial \varphi}{\partial z_k}$ and get since 
$$\overline{\partial}u = \sum_{j<k} \left ( \frac{\partial u_j}{\partial \overline{z}_k}-
\frac{\partial u_k}{\partial \overline{z}_j}\right )\,d\overline{z}_j \wedge d\overline{z}_k$$
that
$$\| \ovprt u \|^2_\varphi + \| \ovprt_\varphi ^* u\|^2_\varphi = 
\int_{\mathbb{C}^n}\sum_{j<k} \left |\frac{\partial u_j}{\partial \overline{z}_k}-
\frac{\partial u_k}{\partial \overline{z}_j}\right |^2 \, e^{-\varphi}\,d\lambda +
\int_{\mathbb{C}^n} \sum_{j,k =1}^n \delta_ju_j\,
\overline{\delta_k u_k}\,  e^{-\varphi}\,d\lambda $$
$$=\sum_{j,k =1}^n \int_{\mathbb{C}^n} \left |\frac{\partial u_j}{\partial \overline{z}_k}
\right |^2 \,e^{-\varphi}\,d\lambda +
\sum_{j,k=1}^n  \int_{\mathbb{C}^n} \left ( \delta_ju_j\,
\overline{\delta_k u_k} - \frac{\partial u_j}{\partial \overline{z}_k}\, 
\overline{\frac{\partial u_k}{\partial \overline{z}_j}}\, \right ) \,e^{-\varphi}\,d\lambda $$
$$=\sum_{j,k =1}^n \int_{\mathbb{C}^n}  \left |\frac{\partial u_j}{\partial \overline{z}_k}
\right |^2 \,e^{-\varphi}\,d\lambda + \sum_{j,k =1}^n  \int_{\mathbb{C}^n} \left [ \delta_j,\frac{\partial}{\partial\overline{z}_k} \right ] \,u_j \, \overline{u}_k\, 
e^{-\varphi}\,d\lambda, $$
where we used the fact that for $f,g\in \mathcal{C}^\infty_0(\mathbb{C}^n)$ we have
$$\left \langle \frac{\partial f}{\partial \overline{z}_k}, g\right \rangle_\varphi =
- \langle f, \delta_kg \rangle_\varphi.
$$
Since
$$ \left [ \delta_j,\frac{\partial}{\partial\overline{z}_k} \right ]= 
\frac{\partial^2 \varphi}{\partial z_j \partial\overline{z}_k},$$

and $\varphi $ satisfies (*) we are done (see also \cite{H}).
\end{proof}
Now it follows by Proposition \ref{basic} that there exists a uniquely determined $(0,1)$-form 
\newline $N_\varphi u\in $dom\,($\ovprt) \,\cap$ dom\,($\ovprt_\varphi^*)$ such that
$$\langle u,v\rangle_\varphi=Q_\varphi (N_\varphi u,v)= \langle\ovprt N_\varphi u,\ovprt v\rangle_\varphi + \langle\ovprt_\varphi^* N_\varphi u,
\ovprt_\varphi^* v\rangle_\varphi,$$
and again by \ref{basic} that
$$\|\ovprt N_\varphi u \|_\varphi ^2 + \|\ovprt_\varphi^* N_\varphi u \|_\varphi ^2 \le  C_1 \|u\|_\varphi ^2$$
as well as
$$\|N_\varphi u\|_\varphi ^2 \le C_2 (\|\ovprt N_\varphi u \|_\varphi ^2 + \|\ovprt_\varphi^* N_\varphi u \|_\varphi ^2) \le  C_3 \|u\|_\varphi ^2, $$
where $C_1, C_2, C_3>0$ are constants. Hence we get that $N_\varphi$ is a continuous linear operator from $L^2_{(0,1)}( \mathbb{C}^n, \varphi)$ into itself (see also \cite{H} or \cite{ChSh}). 
\vskip 1 cm

\section{Weighted Sobolev spaces}~\\

We want to study compactness of the weighted $\ovprt$-Neumann operator $N_\varphi.$ For this purpose we define weighted Sobolev spaces and prove, under suitable conditions,  a Rellich - Lemma for these weighted Sobolev spaces. We will also have to consider their dual spaces, which already appeared in \cite{BDH} and \cite{KM}.

\begin{definition}\label{sob}~\\
For $k\in \mathbb{N}$ let 
$$W^k(\mathbb{C}^n, \varphi):= \{ f\in L^2(\mathbb{C}^n, \varphi) \ : \ D^\alpha f \in 
 L^2(\mathbb{C}^n, \varphi) \ {\text{for}} \ |\alpha | \le k \},$$
 where $D^\alpha =\frac{\partial^{| \alpha|}}{\partial^{\alpha_1}x_1\dots \partial^{ \alpha_{2n}}y_n}$ for $(z_1,\dots ,z_n)=(x_1,y_1,\dots ,x_n,y_n)$ with norm
 $$\|f\|_{k,\varphi}^2= \sum_{|\alpha |\le k}\|D^\alpha f\|_\varphi ^2.$$
 \end{definition}

We will also need weighted Sobolev spaces with negative exponent. But it turns out that for our purposes it is more reasonable to consider the dual spaces of the following spaces.
\begin{definition}\label{dualsob}~\\
Let 
$$X_j=\frac{\partial }{\partial x_j} - \frac{\partial \varphi}{\partial x_j} \ {\text {and}} \ 
Y_j=\frac{\partial }{\partial y_j} - \frac{\partial \varphi}{\partial y_j},$$
for $j=1,\dots ,n$ and define
$$W^1(\mathbb C^n, \varphi ,\nabla \varphi)=\{ f\in L^2(\mathbb{C}^n, \varphi) \ : X_jf,\ Y_jf \in  L^2(\mathbb{C}^n, \varphi) , j=1,\dots ,n \},$$
with norm
$$\|f\|^2_{ \varphi ,\nabla \varphi}= \|f\|^2_{ \varphi }+\sum_{j=1}^n( \|X_jf\|^2_\varphi
+ \|Y_jf\|^2_\varphi) .$$
\end{definition}
In the next step we will analyze the dual space of $W^1(\mathbb C^n, \varphi ,\nabla \varphi).$

By the mapping $f\mapsto (f, X_jf,Y_jf)$, the space $W^1(\mathbb C^n, \varphi ,\nabla \varphi)$ can be identified with a closed product of $L^2(\mathbb{C}^n, \varphi),$ 
hence each continuous linear functional $L$ on $W^1(\mathbb C^n, \varphi ,\nabla \varphi)$ is represented (in a non-unique way) by
$$ L(f)= \int_{\mathbb{C}^n} f(z)g_0(z)e^{-\varphi (z)}\,d\lambda (z) +
\sum_{j=1}^n   \int_{\mathbb{C}^n}(X_jf(z)g_j(z)+Y_jf(z)h_j(z))e^{-\varphi (z)}\,d\lambda (z),$$
for some $g_j,h_j\in L^2(\mathbb{C}^n, \varphi).$

For $f\in \mathcal{C}^\infty_0(\mathbb{C}^n)$ it follows that
$$L(f)= \int_{\mathbb{C}^n} f(z)g_0(z)e^{-\varphi (z)}\,d\lambda (z) -
\sum_{j=1}^n  \int_{\mathbb{C}^n} f(z)\left  ( \frac{\partial g_j(z)}{\partial x_j}+
\frac{\partial h_j(z)}{\partial y_j} \right )e^{-\varphi (z)}\,d\lambda (z) .$$
Since $ \mathcal{C}^\infty_0(\mathbb{C}^n)$ is dense in $W^1(\mathbb C^n, \varphi ,\nabla \varphi)$ we have shown

\begin{lemma}\label{dualsobo}~\\
Each element $u\in W^{-1}(\mathbb C^n, \varphi ,\nabla \varphi):=(W^1(\mathbb C^n, \varphi ,\nabla \varphi))'$ can be represented in a non-unique way by
$$u=g_0+\sum_{j=1}^n \left ( \frac{\partial g_j}{\partial x_j}+
\frac{\partial h_j}{\partial y_j} \right ),$$
where $g_j,h_j\in L^2(\mathbb{C}^n, \varphi).$

The dual norm $\|u\|_{-1,\varphi, \nabla \varphi}:=\sup \{|u(f)| \, : \, \|f\|_{ \varphi ,\nabla \varphi}\le 1\}$ can be expressed in the form
$$\|u\|^2_{-1,\varphi, \nabla \varphi}=\inf \{ \|g_0\|^2+ \sum_{j=1}^n (\|g_j\|^2+\|h_j\|^2 ),$$
where the infimum is taken over all families $(g_j,h_j)$ in $L^2(\mathbb{C}^n, \varphi)$
representing the functional $u$.
\end{lemma}
(see for instance \cite{T})\\
In particular each function in $L^2(\mathbb C^n,\varphi)$ can be indentified with an element of $W^{-1}(\mathbb C^n, \varphi ,\nabla \varphi)$.

\begin{prop}\label{rellich}~\\ 
Suppose that the weight function satisfies
$$\lim_{|z|\to \infty}(\theta |\nabla \varphi (z)|^2+\triangle \varphi (z))= +\infty ,$$
for some $\theta \in (0,1),$ where 
$$|\nabla \varphi (z)|^2= \sum_{k=1}^n \left ( \left | \frac{\partial \varphi}{\partial x_k}\right |^2+ \left | \frac{\partial \varphi}{\partial y_k}\right |^2 \right ).$$
Then the embedding of $W^1(\mathbb C^n, \varphi ,\nabla \varphi)$ into $ L^2(\mathbb{C}^n, \varphi) $ is compact.
\end{prop}
\begin{proof}
We adapt methods from \cite{BDH} or \cite{Jo}, Proposition 6.2., or \cite{KM}. 
For the vector fields $X_j$  from \ref{dualsob} and their formal adjoints $X_j^*=-\frac{\partial}{\partial x_j}$ we have
$$(X_j+X^*_j)f=-\frac{\partial \varphi}{\partial x_j}\, f \ {\text{and}} \ 
[X_j,X^*_j]f= -\frac{\partial^2\varphi}{\partial x_j^2}\, f,$$
for $f\in  \mathcal{C}^\infty_0(\mathbb{C}^n),$ and
$$\langle [X_j,X^*_j]f,f\rangle_\varphi=\|X^*_jf\|^2_\varphi - \|X_jf\|^2_\varphi ,$$
$$\|(X_j+X^*_j)f\|^2_\varphi \le (1+1/\epsilon)\|X_jf\|^2_\varphi + (1+\epsilon)\|X^*_jf\|^2_\varphi $$
for each $\epsilon>0.$ Similar relations hold for the vector fields $Y_j.$
Now we set
$$\Psi (z)=|\nabla \varphi (z)|^2+(1+\epsilon)\triangle \varphi (z).$$

It follows that
$$\langle \Psi f,f\rangle_\varphi \le 
(2+\epsilon +1/\epsilon)\sum_{j=1}^n( \|X_jf\|^2_\varphi
+ \|Y_jf\|^2_\varphi) .$$
Since $\mathcal{C}^\infty_0(\mathbb{C}^n)$ is dense in $W^1(\mathbb C^n, \varphi ,\nabla \varphi)$ by definition, this inequality holds for all $f\in W^1(\mathbb C^n, \varphi ,\nabla \varphi).$

If $(f_k)_k$ is a sequence in $W^1(\mathbb C^n, \varphi ,\nabla \varphi)$ converging weakly to $0,$ then $(f_k)_k$ is a bounded sequence in $W^1(\mathbb C^n, \varphi ,\nabla \varphi)$ and our the assumption implies that 
$$\Psi (z)=|\nabla \varphi (z)|^2+(1+\epsilon)\triangle \varphi (z)$$
is positive in a neighborhood of $\infty $. So we obtain
\begin{align*}
 \int\limits_{\mathbb{C}^n}|f_k(z)|^2e^{-\varphi (z)}\,d\lambda (z) & \le
 \int\limits_{|z|< R}|f_k(z)|^2e^{-\varphi (z)}\,d\lambda (z) +
 \int\limits_{|z|\ge R} \frac{\Psi (z) |f_k(z)|^2}{\inf \{\Psi (z) \, : \, |z|\ge R\}} \, e^{-\varphi (z)}\,d\lambda (z)\\
 &\le  C_{\varphi , R}\, \|f_k\|^2_{L^2(\mathbb B_R)}+ \frac{C_\epsilon \, \|f_k\|^2_{\varphi, \nabla\varphi}}{\inf \{\Psi (z) \, : \, |z|\ge R\}}.
 \end{align*}
Hence the assumption and the fact that the injection $W^1(\mathbb B_R) \hookrightarrow L^2(\mathbb B_R)$ is compact (see for instance \cite{T}) show that a subsequence of $(f_k)_k$ tends to $0$ in $L^2(\mathbb{C}^n, \varphi).$

\end{proof}

\begin{rem}
It follows that the adjoint to the above embedding, the embedding  of $ L^2(\mathbb{C}^n, \varphi) $ into  $(W^1(\mathbb C^n, \varphi ,\nabla \varphi))'=W^{-1}(\mathbb C^n, \varphi ,\nabla \varphi) $ (in the sense of \ref{dualsobo}) is also compact.
\end{rem}
\begin{rem}
Note that one does not need plurisubharmonicity of the weight function in Proposition \ref{rellich}. If the weight is plurisubharmonic, one can of course drop $\theta$ in the formulation of the assumption.
\end{rem}
\vskip 1 cm

\section{Compactness estimates}~\\

The following Proposition reformulates the compactness condition for the case of a bounded pseudoconvex domain in $\mathbb{C}^n$, see \cite{BS}, \cite{Str}. The difference to the compactness estimates for bounded pseudoconvex domains is that here we have to assume an additional condition on the weight function implying a corresponding Rellich - Lemma.

\begin{prop}\label{compact}~\\ 
Suppose that the weight function $\varphi$  satisfies (*) and
$$\lim_{|z|\to \infty}(\theta |\nabla \varphi (z)|^2+\triangle \varphi (z))= +\infty ,$$
for some $\theta \in (0,1),$ then the following statements are equivalent.
\begin{enumerate}
  \item The $\ovprt $-Neumann operator $N_{1,\varphi}$ is a compact operator from $L_{(0,1)}^2(\mathbb{C}^n, \varphi)$ into itself.
  \item The embedding of the space dom\,$(\ovprt )\,\cap$
dom\,$(\ovprt_\varphi^*),$ provided with the graph norm $u\mapsto (\|u\|^2_\varphi + \|\ovprt u\|^2_\varphi + 
\|\ovprt_\varphi ^*u\|^2_\varphi)^{1/2},$ into $L^2_{(0,1)}(\mathbb{C}^n, \varphi)$ is compact.
  \item For every positive $\epsilon $ there exists a constant $C_\epsilon$ such that
  $$\|u\|_\varphi ^2 \le \epsilon (\| \ovprt u \|_\varphi ^2 + \| \ovprt_\varphi ^*u\|_\varphi ^2) + C_\epsilon \|u\|_{-1,\varphi , \nabla \varphi} ^2,$$
  for all $u\in$ dom\,$(\ovprt )\,\cap\,$dom\,$(\ovprt_\varphi^*).$
\item The operators
$$\ovprt_\varphi ^* N_{1,\varphi} : L_{(0,1)}^2(\mathbb{C}^n, \varphi)\cap {\text {ker}}(\ovprt) \longrightarrow L^2(\mathbb{C}^n, \varphi) \ \ {\text {and}}$$
$$\ovprt_\varphi ^* N_{2,\varphi} : L_{(0,2)}^2(\mathbb{C}^n, \varphi)\cap {\text {ker}}(\ovprt) \longrightarrow L_{(0,1)}^2(\mathbb{C}^n, \varphi)$$
are both compact.
\end{enumerate}
\end{prop}
\begin{proof}
First we show that (1) and (4) are equivalent: suppose that $N_{1, \varphi}$ is compact. For $f\in L_{(0,1)}^2(\mathbb{C}^n, \varphi)$ it follows that
$$\| \ovprt_\varphi ^* N_{1, \varphi}f \|_\varphi ^2 \le \langle f,N_{1, \varphi}f\rangle_\varphi \le
\epsilon \|f\|_\varphi ^2 + C_\epsilon \|N_{1, \varphi}\|_\varphi ^2$$
by Lemma 2 of \cite{CD}. Hence $ \ovprt_\varphi ^* N_{1, \varphi}$ is compact.
Applying the formula
$$ N_{1, \varphi}- (\ovprt_\varphi ^* N_{1, \varphi})^*( \ovprt_\varphi ^* N_{1, \varphi})=
( \ovprt_\varphi ^* N_{2, \varphi})( \ovprt_\varphi ^* N_{2, \varphi})^*,$$
(see for instance \cite{ChSh}), we get that  also $\ovprt_\varphi ^* N_{2, \varphi}$ is compact.
The converse follows easily from the same formula.

Now we show (4) $\implies $ (3) $\implies $ (2) $ \implies $ (1). We follow the lines of \cite{Str}, where the case of a bounded pseudoconvex domain is handled.

Assume (4): if (3) does not hold, then there exists $\epsilon_0>0$ and a sequence $(u_n)_n$ in  dom\,$(\ovprt) \,\cap$dom\,$(\ovprt_\varphi^*)$ with $\|u_n\|_\varphi =1$ and
$$\|u_n\|_\varphi ^2 \ge \epsilon_0 (\|\ovprt u_n\|_\varphi ^2 +
\|\ovprt_\varphi ^*u_n\|_\varphi ^2) + n\|u_n\|_{-1,\varphi , \nabla \varphi }^2$$
for each $n\ge 1,$ which implies that $u_n \to 0$ in $W^{-1}_{(0,1)}(\mathbb{C}^n, \varphi, \nabla \varphi) $. Since $u_n$ can be written in the form
$$u_n=(\ovprt_\varphi ^* N_{1,\varphi})^* \,\ovprt_\varphi ^*u_n+
(\ovprt_\varphi ^* N_{2,\varphi })\, \ovprt u_n,$$
(4) implies there exists a subsequence of $(u_n)_n$ converging in $L^2_{(0,1)}(\mathbb{C}^n, \varphi )$ and the limit must be $0, $
which contradicts $\|u_n\|_\varphi =1.$ 

To show that (3) implies (2) we consider a bounded sequence in dom\,$(\ovprt )\,\cap$
dom\,($\ovprt_\varphi^*).$ By \ref{basic} this sequence is also bounded in $L^2_{(0,1)}(\mathbb{C}^n, \varphi ).$ Now \ref{rellich} implies that it has a subsequence converging in 
$W^{-1}_{(0,1)}(\mathbb{C}^n, \varphi, \nabla \varphi) .$ Finally use (3) to show that this subsequence is a Cauchy sequence in $L^2_{(0,1)}(\mathbb{C}^n, \varphi ),$
therefore (2) holds.

Assume (2) : by \ref{basic} and the basic facts about $N_{1,\varphi},$ it follows that
$$N_{1,\varphi}: L^2_{(0,1)}(\mathbb{C}^n, \varphi ) \longrightarrow {\text {dom}}\,(\ovprt )\,\cap
{\text {dom}}\,(\ovprt_\varphi^*)$$
is continuous in the graph topology, hence
$$N_{1,\varphi}: L^2_{(0,1)}(\mathbb{C}^n, \varphi ) \longrightarrow {\text {dom}}\,(\ovprt )\,\cap
{\text {dom}}\,(\ovprt_\varphi^*) \hookrightarrow  L^2_{(0,1)}â(\mathbb{C}^n, \varphi )$$
is compact.

\end{proof}

\begin{rem}\label{sobreg}~\\ 
Suppose that the weight function  $\varphi$ is  plurisubharmonic  and that the lowest eigenvalue $\lambda_{\varphi}$ of the Levi - matrix $M_{\varphi }$ satisfies
$$\lim_{|z|\rightarrow \infty}\lambda_\varphi(z)  = +\infty\, . \ \ \ (^{**})$$
This condition implies that $N_{1,\varphi}$ is compact \cite{HaHe}.

It also implies that  the condition of the Rellich - Lemma \ref{rellich} is satisfied.
\end{rem}
This follows from the fact that we have for the trace ${\text {tr}}(M_\varphi ) $ of the Levi - matrix 
$${\text {tr}}(M_\varphi )=\frac{1}{4}\triangle \varphi, $$
and since for any invertible $(n\times n)$-matrix $T$
$$ {\text {tr}}(M_\varphi )={\text {tr}}(TM_\varphi T^{-1}),$$
it follows that ${\text {tr}}(M_\varphi )$ equals the sum of all eigenvalues of $M_\varphi .$ Hence our assumption on the lowest eigenvalue $\lambda_{\varphi}$ of the Levi - matrix implies that the assumption of Proposition \ref {rellich} is satisfied. 

\medskip
In order to use Propostion \ref{compact} to show compactness of $N_\varphi $ we still need 
\begin{prop}
\label{garding}
\textbf{(G\aa rding's inequality) } Let $\Omega$ be a smooth bounded domain. Then for any $u\in W^1(\Omega , \varphi ,\nabla \varphi)$ with compact support in $\Omega$
\begin{equation*}
\Vert u \Vert ^2_{1,\varphi,\nabla\varphi}\leq C(\Omega, \varphi)\left(\Vert\dquer u\Vert^2_\varphi+\Vert\dquer^*_\varphi u\Vert^2_\varphi +\Vert u\Vert^2_\varphi\right) .
\end{equation*}
\end{prop}

\begin{proof}The operator $-\boxphi$ is strictly elliptic since its principal part equals the Laplacian. Now $-\boxphi=-(\dquer \oplus\dquers)^*\circ(\dquer \oplus\dquers)$, so from general PDE theory follows that the system $\dquer \oplus\dquers$ is elliptic. This is, because a differential operator $P$ of order $s$ is elliptic if and only if $(-1)^sP^*\circ P$ is strictly elliptic. So because of ellipticity, one has on each smooth bounded domain $\Omega$ the classical G\aa rding inequality
\begin{equation*}
\Vert u \Vert ^2_{1}\leq C(\Omega)\left(\Vert\dquer u\Vert^2+\Vert\dquer^*_\varphi u\Vert^2 +\Vert u\Vert^2\right) 
\end{equation*}
for any (0,1)-form $u$ with coefficients in $\mathcal C^\infty_0$. But our weight $\varphi$ is smooth on $\overline \Omega$, hence the weighted and unweighted $L^2$-norms on $\Omega$ are equivalent, and therefore

$$\Vert u \Vert ^2_{1, \varphi , \nabla \varphi} \le C_1 (\Vert u \Vert ^2_{1,\varphi}+\Vert u \Vert ^2_\varphi )
\le C_2  (\Vert u \Vert ^2_1+\Vert u \Vert ^2)$$
$$\le C_3 (\Vert\dquer u\Vert^2+\Vert\dquer^*_\varphi u\Vert^2 +\Vert u\Vert^2)\\
\le C_4 (\Vert\dquer u\Vert^2_\varphi+\Vert\dquer^*_\varphi u\Vert^2_\varphi +\Vert u\Vert^2_\varphi ). $$

\end{proof}

We are now able to give a different proof of the main result in \cite{HaHe}.
\begin{theorem}
\label{cp 1}
Let $\varphi$ be plurisubharmonic. If the lowest eigenvalue $\lambda_\varphi (z)$ of the Levi - matrix $M_\varphi$ satisfies $(^{**})$, then $N_\varphi$ is compact.
\end{theorem}

\begin{proof} By  Proposition \ref{rellich} and Remark \ref{sobreg}, it suffices to show a compactness estimate and use Proposition \ref{compact}. Given $\epsilon>0$ we choose  $M\in\mathbb N$ with $1/M \le \epsilon/2$ and $R$ such that $\lambda(z)>M$ whenever $|z|>R$. Let $\chi$ be a smooth cutoff function identically one on $\mathbb{B}_R$. Hence we can estimate
\begin{align*}
M\Vert f\Vert_\varphi^2\leq& \sum_{j,k} \int\limits_{\mathbb C^n\setminus \mathbb{B}_R}\levim f_j\overline f_k e^{-\varphi}\,d\lambda +M\Vert\chi f\Vert_\varphi^2\\
\leq&Q_\varphi (f,f)+M\langle\chi f,f\rangle_\varphi\\
\leq&Q_\varphi (f,f)+M\Vert\chi f\Vert_{1,\varphi,\nabla\varphi}\Vert f\Vert_{-1,\varphi,\nabla\varphi}\\
\leq&Q_\varphi (f,f)+Ma\Vert\chi f\Vert_{1,\varphi,\nabla\varphi}^2+a^{-1}M\Vert f\Vert_{-1,\varphi,\nabla\varphi}^2 ,
\end{align*}
where $a$ is to be chosen a bit later. Now we apply  G\aa rding's inequality \ref{garding} to the second  term, so there is a constant $C_R$ depending on $R$ such that 
$$M\Vert f\Vert_\varphi^2\leq Q_\varphi (f,f)+MaC_R\left(Q_\varphi (f,f)+\Vert f\Vert_\varphi^2\right)+a^{-1}M\Vert f\Vert_{-1,\varphi,\nabla\varphi}^2.$$
By Proposition \ref{basic} and after increasing $C_R$ we have
$$M\Vert f\Vert_\varphi^2\leq Q_\varphi (f,f)+MaC_RQ_\varphi (f,f)+a^{-1}M\Vert f\Vert_{-1,\varphi,\nabla\varphi}^2.$$
Now choose $a$ such that $aC_R \le \epsilon/2,$ then
\begin{equation*}
\Vert f\Vert_\varphi^2\leq\epsilon Q_\varphi (f,f)+a^{-1}\Vert f\Vert_{-1,\varphi,\nabla\varphi}^2
\end{equation*}
and this estimate is equivalent to compactness by \ref{compact}.
\end{proof}
\vskip 0.5 cm

\begin{rem}
Assumption $(^{**})$ on the lowest eigenvalue of $M_\varphi$ is the analog of property (P) introduced by Catlin in \cite{Ca} in case of bounded pseudoconvex domains. Therefore the proof is similar. 
\end{rem}

\begin{rem} We mention that for the weight $\varphi (z)=|z|^2$ the $\dquer $-Neumann operator fails to be compact (see \cite{HaHe}), but the condition 
$$\lim_{|z|\to \infty}(\theta |\nabla \varphi (z)|^2+\triangle \varphi (z))= +\infty $$
 of the Rellich - Lemma is satisfied.
\end{rem}

\begin{rem}
\label{reg}
Denote by $W^m_{loc}(\mathbb C^n)$ the space of functions which locally belong to the classical unweighted Sobolev space $W^m(\mathbb C^n)$.
Suppose that $\boxphi v=g$ and $g\in W_{loc \, (0,1)}^m(\mathbb C^n)$. Then $v\in W^{m+2}_{loc \, (0,1)}(\mathbb C^n)$. In particular, if there exists a weighted $\dquer$-Neumann operator $N_\varphi$, it  maps $\mathcal{C}_{(0,1)}^\infty(\mathbb C^n)\cap L^2_{(0,1)}(\mathbb C^n,\varphi)$ into itself.
\end{rem}

$\boxphi$ is strictly elliptic, and the statement in fact follows from interior regularity of a general second order elliptic operator. The reader can find more on elliptic regularity for instance in \cite{Ev}, chapter 6.3.\\
An analog statement is true for $\mathcal S_\varphi$. If there exists a continuous canonical solution operator $\mathcal S_\varphi$, it maps $\mathcal{C}_{(0,1)}^\infty(\mathbb C^n)\cap L^2_{(0,1)}(\mathbb C^n,\varphi)$ into itself. This follows from ellipticity of $\dquer$.
\vskip 0.5 cm

Although $\boxphi$ is strictly elliptic, the question whether $\mathcal S_\varphi$ is globally or exactly regular is harder to answer. This is, because our domain is not bounded and neither are the coefficients of $\boxphi$. Only in a very special case the question is easy - this is, when $ A_\varphi^2$ (the weighted space of entire functions) is zero. In this case, there is only one solution operator to $\dquer$, namely the canonical one, and if $f\in W^k_{\varphi\, (0,1)}$  and $u=\mathcal S_\varphi f$, it follows that $\dquer D^\alpha u=D^\alpha f$, since $\dquer$ commutes with $\frac{\partial}{\partial x_j}$. Now $\mathcal S_\varphi$ is continuous,  so $\Vert D^\alpha u\Vert_\varphi\le C\Vert D^\alpha f\Vert_\varphi$, meaning that $u\in W^k_\varphi$. So in this case $\mathcal S_\varphi$ is a bounded operator from $W^k_{\varphi \, (0,1)}\to W^k_\varphi$.
\medskip

\begin{rem}
Let $A^2_{(0,1)}(\mathbb{C}^n, \varphi)$ denote the space of $(0,1)$-forms with holomorphic coefficients belonging to $L^2(\mathbb{C}^n. \varphi ).$

We point out that assuming $(^{**})$ implies directly -- without use of Sobolev spaces -- that the embedding of the space
$$  A^2_{(0,1)}(\mathbb{C}^n, \varphi)\cap dom \,(\dquers )$$ 
provided with the graph norm $u\mapsto (\|u\|^2_\varphi +  
\|\ovprt_\varphi ^*u\|^2_\varphi)^{1/2}$ into $A^2_{(0,1)}(\mathbb{C}^n, \varphi)$ is compact. Compare \ref{compact} (2).
\end{rem}

\vskip 0.3 cm

For this purpose let $u\in A^2_{(0,1)}(\mathbb{C}^n, \varphi)\cap dom \,(\dquers ).$
Then we obtain from the proof of \ref{basic} that
$$\| \dquers u \|^2_\varphi = \int\limits_{\mathbb{C}^n} \sum_{j,k=1}^n\frac{\partial^2 \varphi}{\partial z_j \partial\overline{z}_k}\,u_j\overline{u}_k\, e^{-\varphi}\,d\lambda .$$
Let us for  $u=\sum_{j=1}^n u_j \, d\overline z_j$ indentify $u(z)$ with the vector $(u_1(z),\dots ,u_n(z))\in\mathbb C^n$. Then, if we denote by $\langle.,.\rangle$ the standard inner product in $\mathbb C^n$, we have
$$\langle u(z),u(z)\rangle=\sum_{j=1}^n |u_j(z)|^2 \ {\text{and}} \ 
\langle M_\varphi u(z),u(z) \rangle =  \sum_{j,k=1}^n\frac{\partial^2 \varphi (z)}{\partial z_j \partial\overline{z}_k}\,u_j(z)\overline{u_k(z)} .$$
\vskip 0.3 cm
Note that the lowest eigenvalue $\lambda_\varphi$ of the Levi - matrix $M_\varphi$ can be expressed as 
$$\lambda_\varphi(z) = \inf_{u(z)\neq 0}\frac{\langle M_\varphi u(z),u(z) \rangle}{\langle u(z),u(z) \rangle}.$$
So we get
\begin{align*}
\int_{\mathbb{C}^n}\langle u,u \rangle e^{-\varphi }\,d\lambda \leq&
\int_{\mathbb{B}_R } \langle u,u \rangle e^{-\varphi }\,d\lambda  +[\inf_{\mathbb{C}^n \setminus \mathbb{B}_R} \lambda_\varphi (z)]^{-1} \, \int_{\mathbb{C}^n \setminus \mathbb{B}_R}  \lambda_\varphi (z) \, \langle u,u \rangle e^{-\varphi }\,d\lambda\\
\leq&  \int_{\mathbb{B}_R } \langle u,u \rangle e^{-\varphi }\,d\lambda +
[\inf_{\mathbb{C}^n \setminus \mathbb{B}_R} \lambda_\varphi (z)]^{-1} \, \int_{\mathbb{C}^n}  \langle M_\varphi u,u \rangle e^{-\varphi }\,d\lambda .
\end{align*}

For a given $\epsilon >0 $ choose $R$ so large that
$$[\inf_{\mathbb{C}^n \setminus \mathbb{B}_R} \lambda_\varphi (z)]^{-1}< \epsilon,$$
and use the fact that for Bergman spaces of holomorphic functions the embedding of $A^2(\mathbb{B}_{R_1})$ into $A^2(\mathbb{B}_{R_2})$ is compact for $R_2 <R_1.$ So the desired conclusion follows.

\begin{rem}
Part of the results, in particular Theorem \ref{cp 1}, are taken from \cite{Ga}. We finally mention that the methods used in this paper can also be applied to treat unbounded pseudoconvex domains with boundary, see \cite{Ga}.
\end{rem}

\textbf{Acknowledgement. }The authors thank the referee for two corrections in the bibliography and a remark increasing the readability of the paper. Moreover they are thankful to Emil Straube for pointing out an inaccuracy in an earlier version of the paper and many helpful questions and suggestions.

\end{document}